\documentclass[twoside,11pt]{amsart} 
\usepackage{enumitem}
\usepackage{epsfig}
\usepackage{latexsym} 
\usepackage{amsmath} \usepackage{amssymb}



\newtheorem{theorem}{Theorem}[section]
\newtheorem{lemma}[theorem]{Lemma}
\newtheorem{corollary}[theorem]{Corollary}

\newdimen\epsfxsize
\newdimen\epsfysize

\newcommand {\gap}     {\makebox[0.075 in]{}}   
\newcommand {\biggap}     {\makebox[0.2 in]{}}   
\newcommand {\st}      {\gap : \gap}   
\newcommand{\lil}   {\scriptstyle }
\newcommand{\llil}   {\scriptscriptstyle }

\newcommand {\fto}     {\longrightarrow}

\newcommand {\set}[1]  {\left\{ {#1} \right\}}   
  
\newcommand {\pml}[1]  {{#1}^{\#}}

\newcommand{\pgap}   {{\mathcal G}}

\begin{document}

\title{Expected biases in the distribution of consecutive primes}

\date{28 Apr 2024}

\author{Fred B. Holt}
\address{fbholt62@gmail.com; https://www.primegaps.info}

\begin{abstract}
In 2016 Lemke Oliver and Soundararajan examined the gaps between the first hundred million primes and observed
biases in their distributions modulo $10$.  Given our work on the evolution of the populations of various gaps
across stages of Eratosthenes sieve, the observed biases are totally expected.

The biases observed by Lemke Oliver and Soundararajan are a wonderful example for contrasting the computational range with the asymptotic range 
for the populations of the gaps between primes. The observed biases 
are the combination of two phenomena: 
\begin{itemize}
\item[(a)] very small gaps, say $2 \le g \le 30$, get off to 
quick starts and over the first $100$ million primes larger gaps are too early in their evolution; and 
\item[(b)] the assignment of small gaps across the residue classes disadvantages some of those classes - until
enormous primes, far beyond the computational range.
\end{itemize}

For modulus $10$ and a few other bases, we aggregate the gaps by residue class and track the evolution of these
teams as Eratosthenes sieve continues.  The relative populations across these teams start with biases across
the residue classes.  These initial biases fade as the sieve continues.  
The OS enumeration strongly agrees with a uniform sampling at the corresponding stage of the sieve.  
The biases persist well beyond the computational range,
but they are transient.

\end{abstract}

\maketitle

\section{Setting}

In 2016 Lemke Oliver and Soundararajan \cite{OS, OSQ} studied the distribution of the last digits of consecutive primes, over the first $100$ million 
prime numbers.  They observed a bias in these distributions relative to the simplest expected values.

Our previous work \cite{FBHSFU, FBHPatterns} provides exact models for the populations and relative populations
of gaps $g$ across stages of Eratosthenes sieve.  These models show that the populations of various gaps evolve
very slowly, and this evolution accounts for and predicts the biases observed by Lemke Oliver and Soundararajan \cite{OS, OSQ}.
These biases will persist throughout the computational range, but they will eventually fade away.

\vskip 0.125in

\subsection{Evolution of populations of gaps among primes}
This section is a summary of results in \cite{FBHSFU, FBHPatterns}.  At each stage of Eratosthenes sieve, there is a cycle of gaps $\pgap(\pml{p})$
of length $\phi(\pml{p})$ (number of gaps in the cycle) and span $\pml{p}$ (sum of the gaps in the cycle).  
If we take initial conditions from the cycle of gaps $\pgap(\pml{p_0})$, then we can derive exact models for the
populations $n_{g,j}(\pml{p})$ of driving terms of length $j$ and span $g$ in the cycle $\pgap(\pml{p})$ for all primes $p \ge p_0$
and all gaps $g \le 2p_1$.  For $j=1$, $n_{g,1}(\pml{p})$ denotes the population of the gap $g$ itself in the cycle $\pgap(\pml{p})$.

The populations $n_{g,1}(\pml{p})$ are all superexponential, dominated by the factor $\prod_{1}^{k}(p_i - 2)$.  To facilitate 
comparisons among the gaps, we derive the relative population models $w_{g,j}(\pml{p})$.
$$ w_{g,j}(\pml{p_k}) \; = \; n_{g,j}(\pml{p_k}) / \prod_{p=3}^{p_k} (p - 2) $$
At each stage of the sieve the relative population $w_{g,1}(\pml{p_k})$ represents the superexponential population $n_{g,1}(\pml{p_k})$
as a coefficient on $\prod_{3}^{p_k}(p - 2)$.   These models are derived in \cite{FBHSFU, FBHPatterns}.

\begin{figure}[hbt]
\centering
\includegraphics[width=5in]{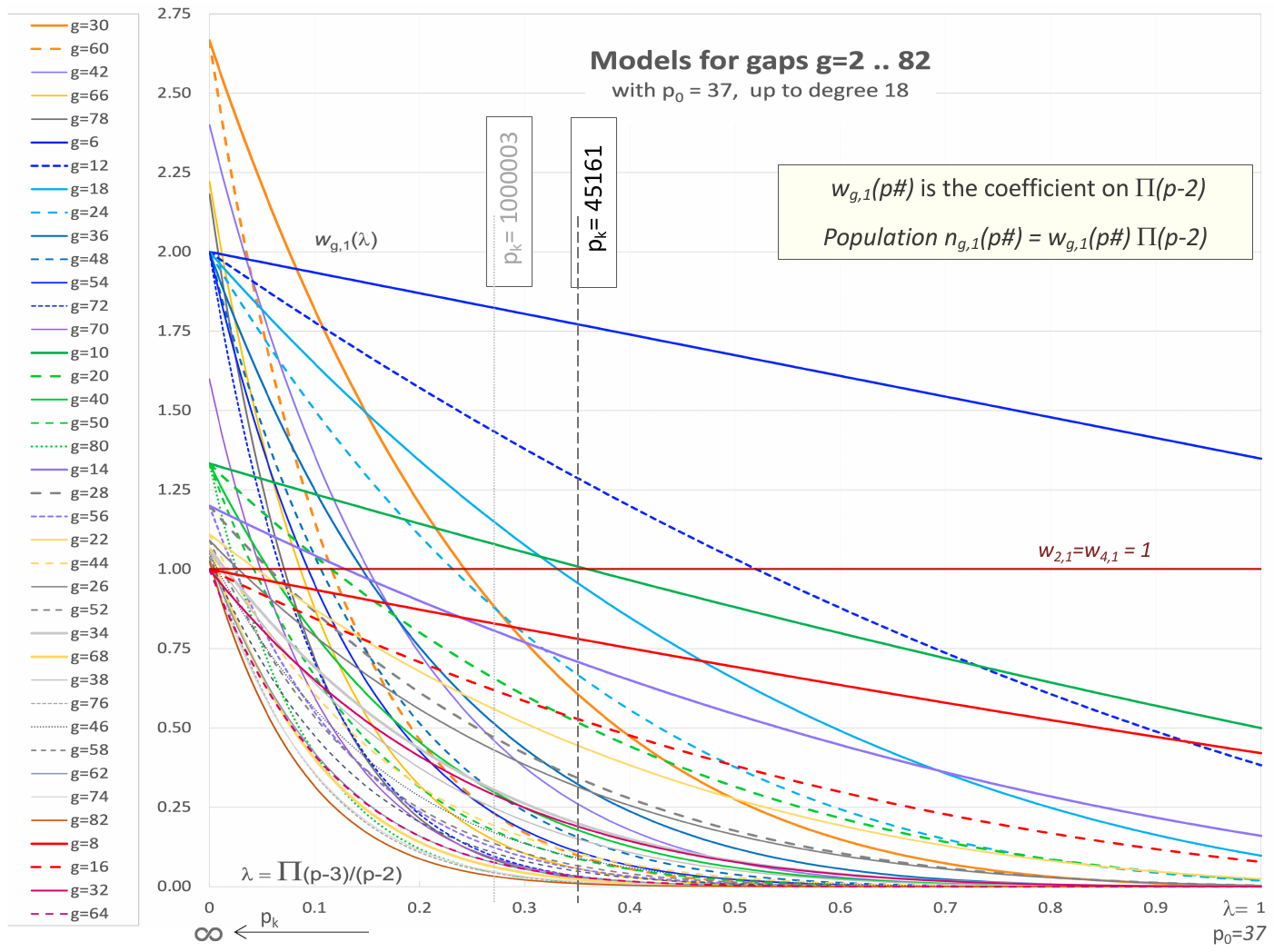}
\caption{\label{AllGapsFig} Shown here are the relative population models $w_{g,1}(\pml{p_k})$ for $p_0=37$ and all gaps $g \le 82$.
The graphs start at the right, where $p=37$ and the system parameter $\lambda=1$.  As the sieve continues, $\lambda \fto 0$
as $p \fto \infty$, and we follow the graphs to the left.  As the sieve proceeds through enormous primes,
the graphs converge toward their asymptotic values at $\lambda=0$.  
Counts for the first $10^8$ primes is fixed in the sieve by
$p_k = 45161$, which occurs around $\lambda=0.348$, and counts up to $x=10^{12}$ are fixed by $p_k=1000003$, where
$\lambda = 0.270$.}
\end{figure}

The relative populations of the gaps $g=2$ and $g=4$ are identically $1$.
$$ w_{2,1}(\pml{p_k}) \; = \; w_{4,1}(\pml{p_k}) \; = \; 1 \biggap \forall p_k > p_0.$$

We use some cycle of gaps $\pgap(\pml{p_0})$ to enumerate the initial populations of the gaps $g \le 2p_1$.  This bound $2 p_1$ encourages us to use
as large a $p_0$ as we can manage.  In our work we use $p_0=37$, so we can derive the exact population models for all gaps $g \le 82$.
New approaches to the enumerations \cite{Brown} of constellations within these large cycles may enable us to increase this starting point to 
$p_0 =41$ or $p_0 = 43$.

The models for the relative populations $w_{g,1}(\pml{p_k})$ for gaps $g \le 82$ are displayed in Figure~\ref{AllGapsFig}.
These models start at the righthand side, with $p_0=37$ and the system parameter $\lambda=1$.  As the sieve continues and $p_k \fto \infty$,
the system parameter 
$$\lambda = \prod_{p_1}^{p_k} \frac{p-3}{p-2} \biggap \fto \biggap 0.$$
The convergence to the asymptotic values $w_{g,1}(\infty)$ is very slow.  We can use Merten's Third Theorem to estimate correspondences between
large values of $p_k$ and small values of $\lambda$.  To reduce $\lambda$ in half, say from $\lambda=c$ to $\lambda=\frac{c}{2}$, we need to 
square the value of the corresponding prime, from $p_k=q$ to $p_k \approx q^2$.

In $\pgap(\pml{37})$, as seen on the righthand side of Figure~\ref{AllGapsFig}, the relative populations of the gaps are ordered primarily by size.
The gap $g=6$, represented by the solid blue line, has already jumped out to be the most populous gap in $\pgap(\pml{p_k})$, and it will continue to 
be the most populous gap until $g=30$ passes it up, when $\lambda \approx 0.083$, which corresponds to $p_k \approx 3.386 E19$.

For any gap $g$, its asymptotic value $w_{g,1}(\infty)$ is determined solely by the prime factors of $g$.  Let $Q$ be the set of odd prime factors
of $g$, then
\begin{equation}
 w_{g,1}(\infty) \; = \; \prod_{q \in Q} \frac{q-1}{q-2}. \label{EqWInf}
 \end{equation}
 These asymptotic values for the relative populations are consistent with the probabilistic predictions of Hardy \& Littlewood \cite{HL, BrentSmall}.
 
On the lefthand side of Figure~\ref{AllGapsFig} we see the gaps sort out into families that share the same sets of odd prime factors $Q$.  
We have color-coded the curves by families sharing the same odd prime factors.  
A closeup for small $\lambda$ and large primes $p_k$ is shown in Figure~\ref{AllGapsZoomFig}.

Treating the cycles of gaps $\pgap(\pml{p})$ as a discrete dynamic system \cite{FBHSFU, FBHPatterns}, we show that the relative population
$w_{g,1}(\pml{p_k})$ of a gap $g$ is given by
\begin{eqnarray*}
w_{g,1} (\pml{p_k}) & = & \prod_{q \in Q} \frac{q-1}{q-2} \; - \; l_2 \cdot {\prod_{p_1}^{p_k} \frac{p-3}{p-2}} \; + \; l_3 \cdot  {\prod_{p_1}^{p_k} \frac{p-4}{p-2}}  \; - \cdots \\
 & & \makebox[1.75 in]{} \cdots + (-1)^{J-1}  l_J \cdot { \prod_{p_1}^{p_k} \frac{p-J-1}{p-2} }\\
 & \approx & \prod_{q \in Q} \frac{q-1}{q-2} \; - \; l_2 \cdot \lambda \; + \; l_3 \cdot \lambda^2 \; - \cdots + \; (-1)^{J-1}  l_J \cdot \lambda^{J-1} 
\end{eqnarray*}
in which $J$ is the length of the longest admissible driving term for $g$, and $Q$ is the set of odd prime factors of $g$.

These models are useful in the following ways.  
For every gap $g$ we can calculate its asymptotic population $w_{g,1}(\infty)$, as given by
Equation~\ref{EqWInf}.
To determine the other coefficients $l_j$, we need the initial conditions (the populations of $g$ and all of its driving terms) for $g$ in some
cycle $\pgap(\pml{p_0})$ for which $g \le 2p_1$.  For gaps within this bound, we can determine the complete models for their relative populations, and
we can use these models to show us the evolution, as in Figure~\ref{AllGapsFig}, far beyond what we could ever enumerate directly.

To put the power of these models in perspective, the cycle $\pgap(\pml{61})$ with $1.54E22$ gaps lies at the horizon of current computing capacity.
The cycle $\pgap(\pml{199})$ has more gaps than there are atoms in the observable universe.  Yet we will work below with results from
$\pgap(\pml{45161})$ and even longer cycles.

\vskip 0.125in

The important observation here is that the relative populations of the various gaps evolve as the sieve continues, and this evolution is
incredibly slow.  When we look at the last digits of consecutive primes, we are assigning the gaps to be on teams by residue class.
Each team's relative population for some $\lambda$ is the sum of the relative populations of all the team's members.

\subsection{Distribution of the last digits for consecutive primes, over the first $100$ million primes.}
Lemke Oliver and Soundararajan \cite{OS} tabulated the counts of consecutive last digits for the first one hundred million primes, and they observed that
the counts were biased away from the distribution that we would expect from the asymptotic values.  
Table~\ref{BiasTable} shows their counts modulo $10$, their relative populations $W_r$ compared to team $r=2$, and the "expected" or asymptotic
values.

\begin{table}[ht]
\begin{center}
\begin{tabular}{c|cr|rcc} \hline
$r \bmod 10$ & $(a,b)$ & $\pi(x_0; 10, (a,b))$ & \multicolumn{1}{c}{$\sigma_r$} & $W_r = \sigma_r /\sigma_2$ & $W_r(\infty)$ \\ \hline
   & $\scriptstyle (1,3)$ & $\scriptstyle 7,429,438$ & & & \\
$2$ & $\scriptstyle (7,9)$ & $\scriptstyle 7,431,870$ & $22,852,739$ & $1$ & $1$ \\
   &  $\scriptstyle (9,1)$ & $\scriptstyle 7,991,431$ & & & \\ \hline
   & $\scriptstyle (3,7)$ & $\scriptstyle 7,043,695$ & & & \\
$4$ & $\scriptstyle (7,1)$ & $\scriptstyle 6,373,981$ & $19,790,617$ & $0.8660$ & $1$ \\
   &  $\scriptstyle (9,3)$ & $\scriptstyle 6,372,941$ & & & \\ \hline
   & $\scriptstyle (1,7)$ & $\scriptstyle 7,504,612$ & & & \\
$6$ & $\scriptstyle (3,9)$ & $\scriptstyle 7,502,896$ & $21,762,703$ & $0.9523$ & $1$ \\
   &  $\scriptstyle (7,3)$ & $\scriptstyle 6,755,195$ & & & \\ \hline
   & $\scriptstyle (1,9)$ & $\scriptstyle 5,442,345$ & & & \\
$8$ & $\scriptstyle (3,1)$ & $\scriptstyle 6,010,982$ & $17,466,066$ & $0.7643$ & $1$ \\
   &  $\scriptstyle (9,7)$ & $\scriptstyle 6,012,739$ & & & \\ \hline
   & $\scriptstyle (1,1)$ & $\scriptstyle 4,623,042$ & & & \\
$0$ & $\scriptstyle (3,3)$ & $\scriptstyle 4,442,562$ & $18,127,875$ & $0.7932$ & $4 / 3$ \\
   & $\scriptstyle (7,7)$ & $\scriptstyle 4,439,355$ & & & \\
   &  $\scriptstyle (9,9)$ & $\scriptstyle 4,622,916$ & & & \\ \hline
\end{tabular}
\caption{\label{BiasTable} Lemke Oliver and Soundararajan's enumeration of the distributions
of last digits of consecutive primes for the first $10^8$ primes.  They observe a bias $W_r$ compared to
the values we might expect from the asymptotic values $W_r(\infty)$.}
\end{center}
\end{table}

The notation $\pi(x_0; B, (a,b))$ stands for the number of pairs of consecutive primes $p_k < p_{k+1}$ such that $p_{k+1} < x_0$, and in base $B$
the last digit of $p_k$ is $a$ and the last digit of $p_{k+1}$ is $b$.  We group the pairs $(a,b)$ by their residues ${r = (b-a)\bmod B}$.

The simplest expectation is that each ordered pair $(a,b)$ should occur equally often.  Since the class $r=0$ has four ordered pairs $(a,b)$
assigned to it and the other classes each have three, we would expect the class $r=0$ to occur $4/3$ times as often as any other class.
In Table~\ref{BiasTable} we see that over the first $100$ million primes the relative occurrences $W_r$ are significantly different from
the expected values $W_r(\infty)$.

We will provide evidence that the asymptotic values $W_r(\infty)$ do indeed meet the simple expected values.  We will also show that the observed biases
are consistent with the relative populations of the gaps in Figure~\ref{AllGapsFig} at the point in the evolution covering the first $100$ million primes.

\section{Expected populations by residue class}
We assign each gap $g$ onto its respective team $r \bmod 10$.  We will see that some teams, notably $r=8$ and $r=0$, are disadvantaged
in the early stages of the evolution of gaps.  There simply are too few gaps on these teams up through the first $100$ million odd primes.
These biases will fade away, but this will occur well beyond the computational range.

\subsection{Interval of survival $\Delta H(p)$.}
The cycle of gaps $\pgap(\pml{p})$ has span $\pml{p}$, and we are led to consider where this cycle might best be reflected among the primes
themselves.  For the cycle $\pgap(\pml{p_k})$, all of the gaps up through $p_{k+1}^2$ will be confirmed as gaps among primes.  That is, 
$p_{k+1}^2$ is the smallest composite number remaining after Eratosthenes sieve has run through $p_k$.
On the other hand, all of the gaps up through $p_k^2$ were fixed by previous cycles.  So we define the {\em interval of survival} for the
cycle of gaps $\pgap(\pml{p_k})$ to be the interval
$$ \Delta H(p_k) \; = \; [p_k^2, \; p_{k+1}^2]. $$
To the extent that the gaps are distributed uniformly in $\pgap(\pml{p})$, we would expect the distribution of gaps in $\Delta H(p)$ to 
reflect this.

The $100$ millionth odd prime is ${x_0 = 2,038,074,751}$.  The prime ${p=45161}$ is the next prime larger than $\sqrt{x_0}$, and so the
counts of the first $100$ million gaps are covered by the intervals of survival up to $\Delta H(45161)$.  At $p_k=45161$ all of the gaps up through
$x_0$ have been confirmed as gaps among primes.  No composite numbers are left in the interval $[3,x_0]$ covered by the enumeration.

When we use $p_0=37$ to set the initial conditions for $w_{g,1}(\pml{p})$, 
the value $p_k=45161$ corresponds to $\lambda=0.3481$.  We have marked this section
across the relative population models in Figure~\ref{AllGapsFig} above.  Notice how few gaps have relative populations
$w_{g,1}(p_k)  > 0.25$ at this point and how far these are from their asymptotic populations at $\lambda=0$.

\subsection{Gaps of same residue $r \bmod 10$.}
To connect the cycles of gaps $\pgap(\pml{p})$ to the study of last digits of consecutive primes, we assign the gaps to teams by
their residues $r = g \bmod 10$.

For example, if consecutive primes are in the set $(a,b)=(3,1)$, then the gap between them must be $g=8$ or $g=18$ or $g=28$ or 
some other gap $g = 8 \mod 10$.  

The ordered pairs of last digits $(a,b)$ of consecutive primes
correspond to gaps between primes in the following way.
\begin{center}
\begin{tabular}{clcl}
$r \bmod 10$ & \multicolumn{1}{c}{$(a,b)$'s} & $\Leftrightarrow $ & \multicolumn{1}{c}{$g$'s} \\ \hline
$2$ & $\lil (1,3), \; (7,9), \; (9,1)$ & & $2, 12, 22, 32, 42, \ldots $ \\
$4$ & $\lil (3,7), \; (7,1), \; (9,3)$ & & $4, 14, 24, 34, 44, \ldots $ \\
$6$ & $\lil (1,7), \; (3,9), \; (7,3)$ & & $6, 16, 26, 36, 46, \ldots $ \\
$8$ & $\lil (1,9), \; (3,1), \; (9,7)$ & & $8, 18, 28, 38, 48, \ldots $ \\
$0$ & $\lil (1,1), \; (3,3), \; (7,7), \; (9,9)$ & & $10, 20, 30, 40, 50, \ldots $ 
\end{tabular}
\end{center}

\vskip 0.125in

We form a partial sum $\sigma_r$ for each residue class
modulo $10$, and then we take the ratios to $\sigma_2$, defining $W_r = \sigma_r / \sigma_2$.
$$ \sigma_r (x) \; = \; \sum_{b-a = r \bmod 10} \pi(x; 10, (a,b)) \; = \; \sum_{g = r \bmod 10} n_{g,1}(x),$$
and
$$ W_r(x) \; = \; \frac{\sigma_r(x)}{\sigma_2(x)} \; = \; {\sum_{g=r \bmod 10} w_{g,1}(x)}\; / {\sum_{g=2 \bmod 10} w_{g,1}(x)}.$$

For the partial sums, we limit the sums by the range of gaps, ${2 \le g \le g_{\max}}$,
and by the cycle of gaps $\pgap(\pml{p})$ over which we perform the comparative enumeration.
$$\sigma_r(g_{\max},\pml{p}) = \sum_{\substack{g \equiv r \bmod 10; \\ g \le g_{\max}}} w_{g,1}(\pml{p}) $$
Consistent with our normalized populations above, we then form the ratios
$$ W_r(g_{\max},\pml{p}) \; = \; \frac{\sigma_r(g_{\max},\pml{p})}{\sigma_2(g_{\max}, \pml{p})}. $$

We are interested in the limits $W_r(\infty, \infty)$ if they exist, and in the evolution of these
ratios, especially around $p_k = 45161$. 

\vskip 0.125in

This is a wonderful example for contrasting the computable range with the asymptotic range 
for the relative populations of gaps in $\pgap(\pml{p})$. 
The biases observed by Lemke Oliver and Soundararajan
are the combination of two phenomena: 
\begin{itemize}
\item[(a)] very small gaps, say $2 \le g \le 30$, get off to very
quick starts and for $\lambda \approx 0.348$ larger gaps are too early in their evolution, and 
\item[(b)] the assignment of small gaps across the residue classes disadvantages some of those classes - until
$\lambda$ gets very small.
\end{itemize}

\subsection{a.  A few small gaps get quick starts.}
The righthand side of Figure~\ref{AllGapsFig} illustrates this phenomenon of small gaps getting
rapid headstarts in $w_{g,1}(\pml{p})$. 

We have $w_{2,1}(\pml{p}) = w_{4,1}(\pml{p}) = 1$ for all $p \ge 3$, and this level of
$w_{g,1}=1$ is one yardstick for comparing the relative populations of other gaps.

When $p=37$, which is $\lambda=1$ in Figure~\ref{AllGapsFig}, we see that the gap $g=6$ is the most populous
gap in $\pgap(\pml{p})$ and already has a relative population of $w_{6,1}(\pml{37})=1.35$.
The only other gap cresting $w_{g,1}=0.5$ is $g=10$.

Over the first $100$ million primes, with $\lambda =0.348$, we see that $w_{6,1}(\pml{45161})=1.77$ and
only three other gaps are starting to cross that threshold $w_{g,1}=1$, the gaps $g=12, \; 18,\; 10$.
Only $17$ of the $41$ gaps from $g=2,\ldots,82$ have values $w_{g,1}(\pml{45161}) > 0.25$, 
and all $17$ of these gaps satisfy $g \le 42$.

From Equation~\ref{EqWInf} we see that the minimum asymptotic value is $w_{2^k,1}(\infty)=1$.  Thirteen of the gaps $g \le 82$ have
asymptotic values $w_{g,1}(\infty) \ge 2$, but even at $\lambda=0.2$ only five of these thirteen
will have passed above the threshold $w_{g,1} > 1$, and this corresponds to $p \approx 127,292,773$,
whose interval of survival extends beyond $1.6E16$.

The gap $g=30$ finally passes $g=6$ as the most populous gap at $\lambda=0.083$ or primes
$p_k \approx 3.386E19$, with a relative population of $1.946$.  
The gaps $g=210$ and $g=2310$ will
become more populous in turn, but we cannot even get initial conditions for these gaps until cycles
well beyond our reach.  We can calculate the asymptotic values for all of the gaps but not the models 
for the evolution of their populations.

The point here is that through and well beyond the computational range, a few small gaps dominate the
distributions, and larger gaps do not get anywhere near the asymptotic values for their relative populations
until much larger primes.
By its nature, the recursion that produces the cycles of gaps $\pgap(\pml{p})$ introduces the
gaps $g$ into the cycles approximately in order of magnitude, with occasional localized exceptions.

\subsection{b. The assignment of small gaps across the teams disadvantages some teams.}
Let's look at the rosters of small gaps assigned to each team $r \bmod 10$.  
Refer to Figure~\ref{AllGapsFig} to check the values in the following.

The gaps $g=6$ and $g=12$ are the only two gaps contributing significantly more than $w_{g,1}=1$ at
$\lambda=0.348$, so the teams $W_6$ and $W_2$ have this advantage.

Then there's a cluster of gaps around $w_{g,1}=1$:  $g=2,4,10,18$.  Falling from $w_{g,1}=0.75$ down past 
$w_{g,1}=0.5$, we have in order the gaps ${g= 8,14, 24,30,16,20,22}$.  The distribution of these 
contributions across the residue classes anchors the early bias.  We include the gaps $g=34,40,48$ in this summary
so that each residue class is represented by its top four members at $\lambda=0.348$.

\begin{center}
\begin{tabular}{c|ccccccc|r|}
$\lil r \bmod 10$ & $\lil g=6$ & $\lil g=12$ & $\lil 10,2,4,18$ & $\lil 8,14,24,30$ & $\lil 16,20,22$ & $\lil 28,36,26,42$ & $\lil 34,40,48$ & $\lil \sum_r$ \\ \hline
$2$ &   & $\lil 1.29 $ & $\lil  1$ &  & $\lil 0.45$  & $\lil 0.27$ & & $\lil 3.01$ \\
$4$ &  &  & $\lil  1$ & $\lil 0.71+0.67$ &   & & $\lil 0.20$ & $\lil 2.58$  \\
$6$ & $\lil 1.77 $  &  &  &  & $\lil 0.53$  & $\lil  0.33+0.32$ & & $\lil 2.95$ \\
$8$ &   &  & $\lil 0.96 $ & $\lil 0.78$ &   & $\lil 0.35$ & $\lil 0.16$ & $\lil 2.25$ \\
$0$ &   &  & $\lil 1.01 $ & $\lil 0.61$  & $\lil 0.52$  & & $\lil 0.18$ & $\lil 2.32$  
\end{tabular}
\end{center}

\begin{figure}[htb]
\centering
\includegraphics[width=5in]{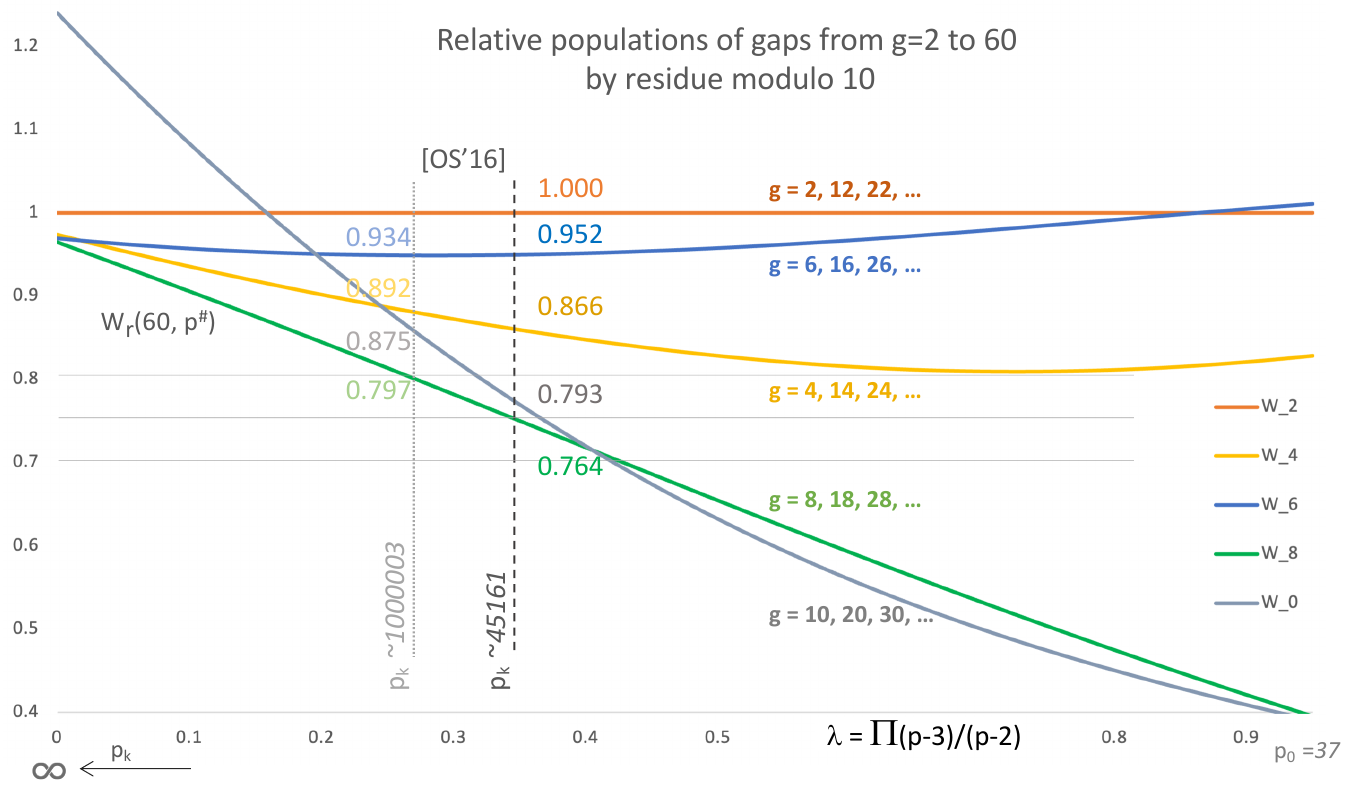}
\caption{\label{WsFig} The ratios $W_r(60,\pml{p})$ of the populations of gaps in each residue class modulo $10$,
normalized by the population of gaps $g = 2 \bmod 10$.  The initial conditions for gaps ${2 \le g \le 60}$ come from $\pgap(\pml{37})$.  
The dashed lines indicate where the calculations by Lemke Oliver and Soundararajan lie, for the first $100$ million primes and for $x_0=10^{12}$.}
\end{figure}

Figure~\ref{WsFig} shows the relative populations of the five residue classes, with gaps $g_{\max} \le 60$.
We mark the section for $p_k = 45161$ at $\lambda = 0.348$.  This covers the enumeration by Lemke Oliver and Soundararajan.
The biases they observed are starkly represented in these models.

We see that the teams $8 \mod 10$ and $0 \bmod 10$ are heavily disadvantaged, since their small gaps are relatively
scarce in $\pgap(\pml{p})$ through the computational range.

If we have models up to $g=82$ why do we stop at $g_{\max}=60$?  To keep the team rosters as fair as we can.
As we increase $g_{\max}$, we quickly see that it is only fair to increase it by $10$ at a time, so that 
the classes have equal numbers of member gaps.  
Gaps that are multiples of $6$ add a weight of at least $2$ to $W_r(\infty)$.  When we consider the multiples of $6$, these
are distributed across the five residue classes on a cycle that has a period of $30$ for $g_{\max}$.
If $g_{\max}$ is a multiple of $30$, each team has the same total number of gaps and same number of multiples of $6$.  
With models up to $g \le 82$, the best we can do is $g_{\max}=60$.  

If we could get initial conditions from $\pgap(\pml{43})$, we would determine the models for the relative populations for gaps
up to $g \le 94$ and we could raise $g_{\max}$ for complete models to $90$.
 
 \vskip 0.125in
 
In $W_r(g_{\max},\pml{p})$
the two parameters $g_{\max}$ and $p$ are not independent.
For any choice of $p$ there is a maximum gap in the cycle $\pgap(\pml{p})$, setting an upper
bound on an effective choice of $g_{\max}$.  Further, among the gaps that do occur in $\pgap(\pml{p})$
the populations of the larger gaps severely lag the populations of the smaller gaps.
So the slices of $W_r(g_{\max},\pml{p})$ in which we fix $p$ and let $g_{\max}$ increase are of interest
at most until $g_{\max}$ reaches the size of the maximum gap in $\pgap(\pml{p})$.

We also investigate the slices of $W_r(g_{\max}, \pml{p})$ in which we fix
$g_{\max}$ and let $p \fto \infty$.  When we fix $g_{\max}$ we have to remain aware that for 
very large values of $p$, we are ignoring the evolution of populations of larger gaps.
Figure~\ref{RhoFig} shows that over the computational range the gaps $g \le 82$ account for 
well over $80\%$ of the total gaps.

Table~\ref{TblWr} provides the asymptotic values $W_r(g_{\max},\infty)$ for a few $g_{\max}$ that are multiples
of primorials.  The table also includes examples for ${g_{\max}=80, 82}$ of the deviations introduced by intermediate
values.  These deviations in the $W_r$ will diminish as $g_{\max}$
gets larger.  

\begin{table}[h]
\begin{center}
\begin{tabular}{r|cllll}
\multicolumn{1}{c}{$g_{\max}$} &  \multicolumn{1}{c}{$\lil W_2(g_{\max} , \infty)$} &
  \multicolumn{1}{c}{$\lil W_4(g_{\max} , \infty)$} &  \multicolumn{1}{c}{$\lil W_6(g_{\max} , \infty)$} &
    \multicolumn{1}{c}{$\lil W_8(g_{\max} , \infty)$} &  \multicolumn{1}{c}{$\lil W_0(g_{\max} , \infty)$}  \\ \hline
$30$ & $1$ & $\lil 1.02162162$ & $\lil 0.99508600$ & $\lil 1.02162162$ & $\lil 1.29729730$ \\
$60$ & $1$ & $\lil 0.97393142$ & $\lil 0.96936858$ & $\lil 0.96440840$ & $\lil 1.24001879$ \\
$80$ & $1$ &  $\lil 0.89428493$ & $\lil 0.99854518$ & $\lil 0.99208031$ & $\lil 1.16873601$ \\
 $82$ & $1$ &  $\lil 0.82184733$ & $\lil 0.91766244$ &  $\lil 0.91172122$ & $\lil 1.07406771$ \\
$90$ & $1$ & $\lil 1.01138869$ & $\lil 0.99856424$ & $\lil 0.99947185$ & $\lil 1.28466922$ \\ \hline
$210$ & $1$ & $\lil 1.00041604$ & $\lil 1.00223358$ & $\lil 1.00193536$ & $\lil 1.30766905$ \\
$420$ & $1$ & $\lil 1.00072633$ & $\lil 1.00293216$ & $\lil 1.00258867$ & $\lil 1.31916484$ \\ 
$630$ & $1$ & $\lil 0.99896438$ & $\lil 0.99872360$ & $\lil 0.99867563$ & $\lil 1.31777582$ \\ \hline
$2310$ & $1$ & $\lil 0.99994434$ & $\lil 0.99977058$ & $\lil 0.99972744$ & $\lil 1.32869087$ \\
$30030$ & $1$ & $\lil 0.99996112$ & $\lil 0.99998697$ & $\lil 0.99998873$ & $\lil 1.33276811$ \\ 
$\infty$ & $1$ & $1$ & $1$ & $1$ & $1.33333$ \\ \hline
\end{tabular}
\end{center}
\caption{ \label{TblWr}This table shows the evolution of $W_r(g_{\max},\infty)$ over a modest range
of $g_{\max}$.  Gaps that are multiples of primorials, $g=k\cdot \pml{p}$, provide good markers
for the trends of $W_r$, because these gaps allow the distribution of prime factors to even out across 
the classes.  Intermediate values, represented here by the rows $g_{\max}=80$ and $82$ produce predictable deviations.  
The gap $g=82$ contributes $w_{82,1}(\infty)=1.026$ toward $W_2$.  This increases the denominator
for the other $W_r$ and lowers these values.  As $g_{\max}$ increases the asymptotic ratios $W_r(g_{\max}, \infty)$ 
approach the simple expected values.}
\end{table}

When we take $g_{\max}$ to be a multiple of the primorial $\pml{p_k}$, 
the multiples of primes up through $p_k$ are distributed
uniformly across the residue classes, except of course for the multiples of $p=5$, all of which fall in the
class $0 \bmod 10$.  For each even gap $g$ that is not a multiple of $5$, the gaps $g$, $2g$, $3g$, $4g$
occur in different residue classes ${r=2,4,6,8}$.
The gaps $g$, $2g$, and $4g$ all contribute the same weight $w_g(\infty)$ to their residue classes $W_r$.
The gap $3g$ has weight $2w_g(\infty)$ unless $3 | g$. 
And the gap $5g$ contributes $\frac{4}{3}w_g(\infty)$ to $W_0$.

If $g_{\max}$ is chosen between multiples of $\pml{p_k}$, then multiples of some primes,
including at least the multiples of $p_k$, 
will not have completed a cycle through the residue classes.  If a cycle is incomplete, the residue class $0 \bmod 10$ is always
disadvantaged since it is always the last class to receive a member in these cycles.

Consequently, in Figure~\ref{WsFig}, we use $g_{\max}=60$ even though we have
complete models for gaps up to $g=82$. 

The incomplete cycles for primes $q > p_k$ perturb the values $W_r$ away from their asymptotic values.
These perturbations decrease for larger $q$, since the prime $q$ contributes a factor $\frac{q-1}{q-2}$
to the terms corresponding to the gaps $g$ that are multiples of $q$.  As $q$ grows, this factor gets closer
to $1$.

\subsection{Computational horizon for primes.}
It is interesting to note here again how remote the asymptotic state is.
Based on these asymptotic distributions, we expect that the biases observed by 
Lemke Oliver and Soundararajan will disappear among large primes.  

For the cycles of gaps $\pgap(\pml{p})$, the evolution of the dynamic system
plays out on massive scales.  Compared to the scales on which the primes evolve, 
Lemke Oliver and Soundararajan's 
sample of the first $10^8$ primes is very early in the evolution of prime gaps.
Their sample is initially and amply covered by the cycle of gaps $\pgap(\pml{29})$,
and the sample falls within the $p_{k+1}^2$ horizon for survival for $\pgap(\pml{45161})$.
The dashed line in Figure~\ref{WsFig} marks the spot, around $\lambda \approx 0.348$,
where the ratios for the Lemke Oliver and
Soundararajan data occur for the partial sums $W_r(60,\pml{p})$.

\begin{figure}[htb]
\centering
\includegraphics[width=4.5in]{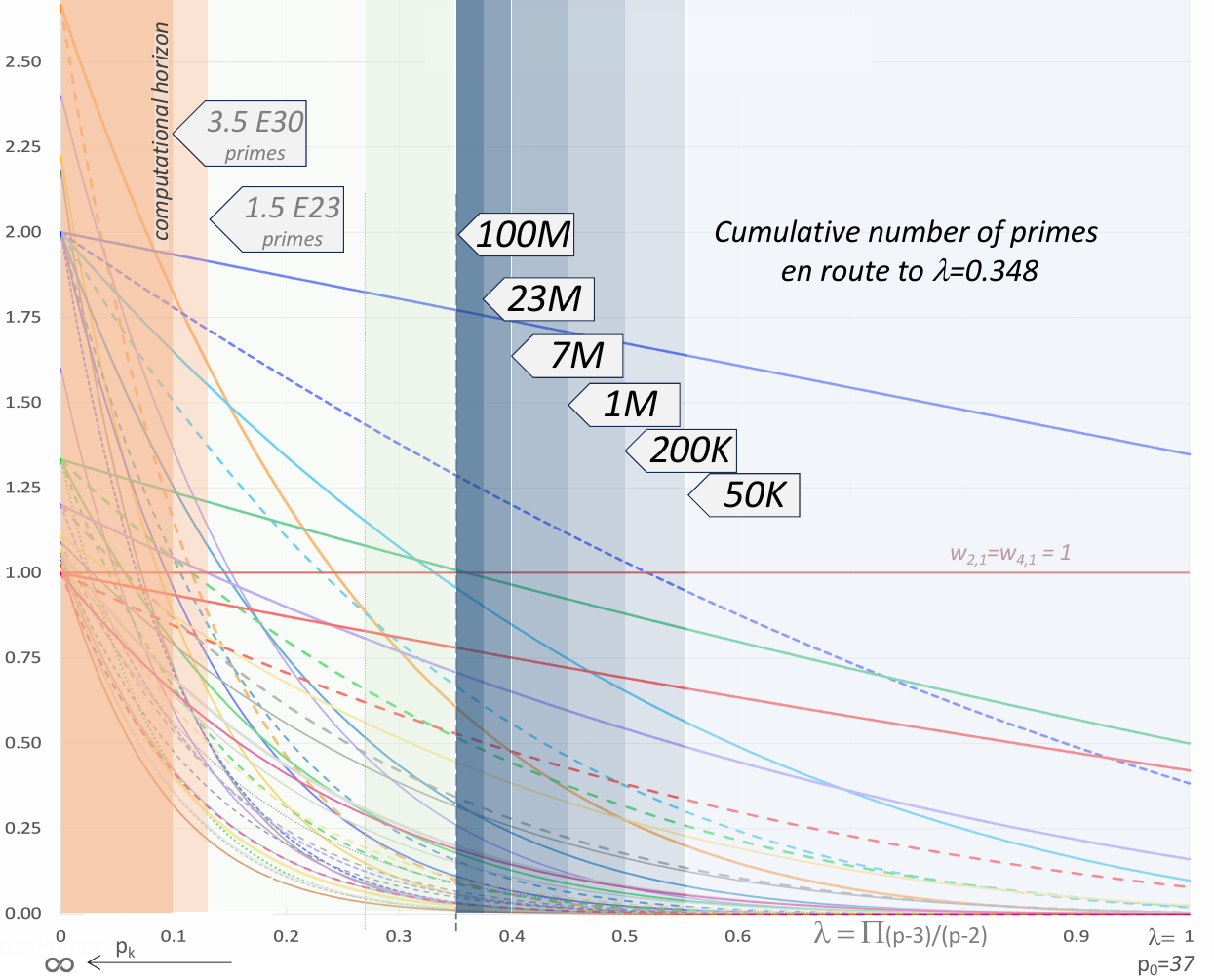}
\caption{\label{100MFig} Showing the accumulation of the first $1 E8$ primes from $\lambda=1$ to $0.348$.  We see that the relative
populations around $\lambda \approx 0.35$ to $0.4$ dominate this collection of gaps.  We also mark the horizon for the computational range.}
\end{figure}

Figure~\ref{100MFig} illustrates the distribution of the first $100$ million primes across the models of the relative populations of gaps $g \le 82$.
We see that $77$ million primes lie between $\lambda = 0.348$ and $0.375$,  and that $93$ million primes lie between $\lambda=0.348$
and $0.4$. This bolsters the expectation that the overall statistics for the first $100$ million primes would track closely with the relative 
populations $w_{g,1}$ over these intervals.

How far can we expect to conduct other large statistical samples like the Lemke Oliver \& Soundararajan data?
In Figure~\ref{100MFig} we also mark the computational horizon, represented here by the first $1.5 E23$ primes and the first $3.5 E30$ primes.
The first value indicates the current estimate for all data storage globally.  The second value is the number of primes accumulated
on the way down to $\lambda=0.1$.  

Any exhaustive statistical search, for example over a complete interval of survival $\Delta H(p)$, is likely to be in the range $\lambda > 0.1$.
The size of data and the computational time required forestall getting complete samples beyond this horizon. 
Note that this horizon occurs before the gap $g=30$ has become the most populous gap and long before gaps with even larger asymptotic relative 
populations will proliferate.  Computational samples of the primes do not reflect their asymptotic behaviors.

Another question about the completeness and accuracy of our analysis above is about the gaps that are not included, the gaps $g > 82$.
For a given value of $\lambda$, what proportion $\rho(\lambda)$ of all the gaps are represented by our models of the relative populations of the gaps
$g \le 82$?

\begin{lemma}
Let ${\mathcal S}$ be a subset of gaps and let $\rho_{\mathcal S}(\lambda)$ be the proportion of all gaps represented by the gaps
in ${\mathcal S}$ at the value $\lambda$.  Then for $p > 1500$
\begin{equation}
\rho_{\mathcal S}(\lambda) \approx c_0 \cdot \lambda \cdot  \sum_{g \in {\mathcal S}} w_{g,1}(\pml{p})
\end{equation}
with $c_0 = 0.6099 \cdot  \prod_5^{p_0} \frac{q-3}{q-2}$
\end{lemma}

\begin{proof}
The total number of gaps in a cycle $\pgap(\pml{p})$ is $\phi(\pml{p})$.  Let $p_0$ be the prime at which $\lambda =1$, and let $p_k$ be
a prime corresponding to the value $\lambda$.
We can estimate $\rho$  
\begin{eqnarray*}
 \rho_{\mathcal S}(\lambda) & = & \frac{\sum_{g \in {\mathcal S}} w_{g,1}(\lambda)}{\sum_{g} w_{g,1}(\lambda)} \; = \;
   \frac{\sum_{g \in {\mathcal S}} w_{g,1}(\lambda)}{\prod_{q=3}^{p_k} \frac{q-1}{q-2}} \\
     & = &  \frac{1}{2} \cdot \prod_{q=5}^{p_k} \frac{(q-2)(q-2)}{(q-1)(q-3)} \cdot \prod_{q=5}^{p_0}\frac{q-3}{q-2} \cdot \prod_{q=p_1}^{p_k} \frac{q-3}{q-2} 
     \cdot \sum_{g \in {\mathcal S}} w_{g,1}(\lambda)\\
     & = &  \frac{1}{2} \cdot \prod_{q=5}^{p_k} \left( 1 + \frac{1}{(q-1)(q-3)} \right) \cdot \prod_{q=5}^{p_0}\frac{q-3}{q-2} \cdot \lambda(p_k) 
     \cdot \sum_{g \in {\mathcal S}} w_{g,1}(\lambda)\\
     & \approx &  0.6099 \cdot \prod_{q=5}^{p_0}\frac{q-3}{q-2} \cdot \lambda(p_k) 
     \cdot \sum_{g \in {\mathcal S}} w_{g,1}(\lambda)
\end{eqnarray*}
The approximation for the first product holds for $p_k > 1500$.
\end{proof}

\begin{corollary}
When $p_0 = 37$ and ${\mathcal S} = \set{g \st g \le 82}$, 
$$\rho(\lambda) \approx 0.1975 \cdot \lambda \cdot \sum_{g \le 82} w_{g,1}(\lambda).$$
\end{corollary}

\begin{figure}[htb]
\centering
\includegraphics[width=4.75in]{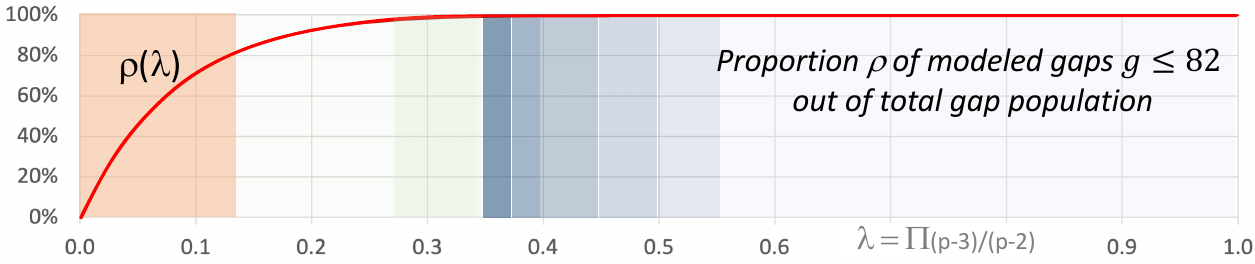}
\caption{\label{RhoFig} The graph of the proportion $\rho$ of modeled gaps $g \le 82$ to all gaps in $\pgap(\pml{p})$. 
Over the OS sample ${\rho > 99\%}$, and
throughout the computational range ${\rho > 80\%}$.}
\end{figure}

We graph $\rho(\lambda)$ in Figure~\ref{RhoFig}.  We see that larger gaps $g > 82$ do not impact our current results and account for less than
$20\%$ of total gaps through the computational range.

\begin{figure}[hbt]
\centering
\includegraphics[width=5in]{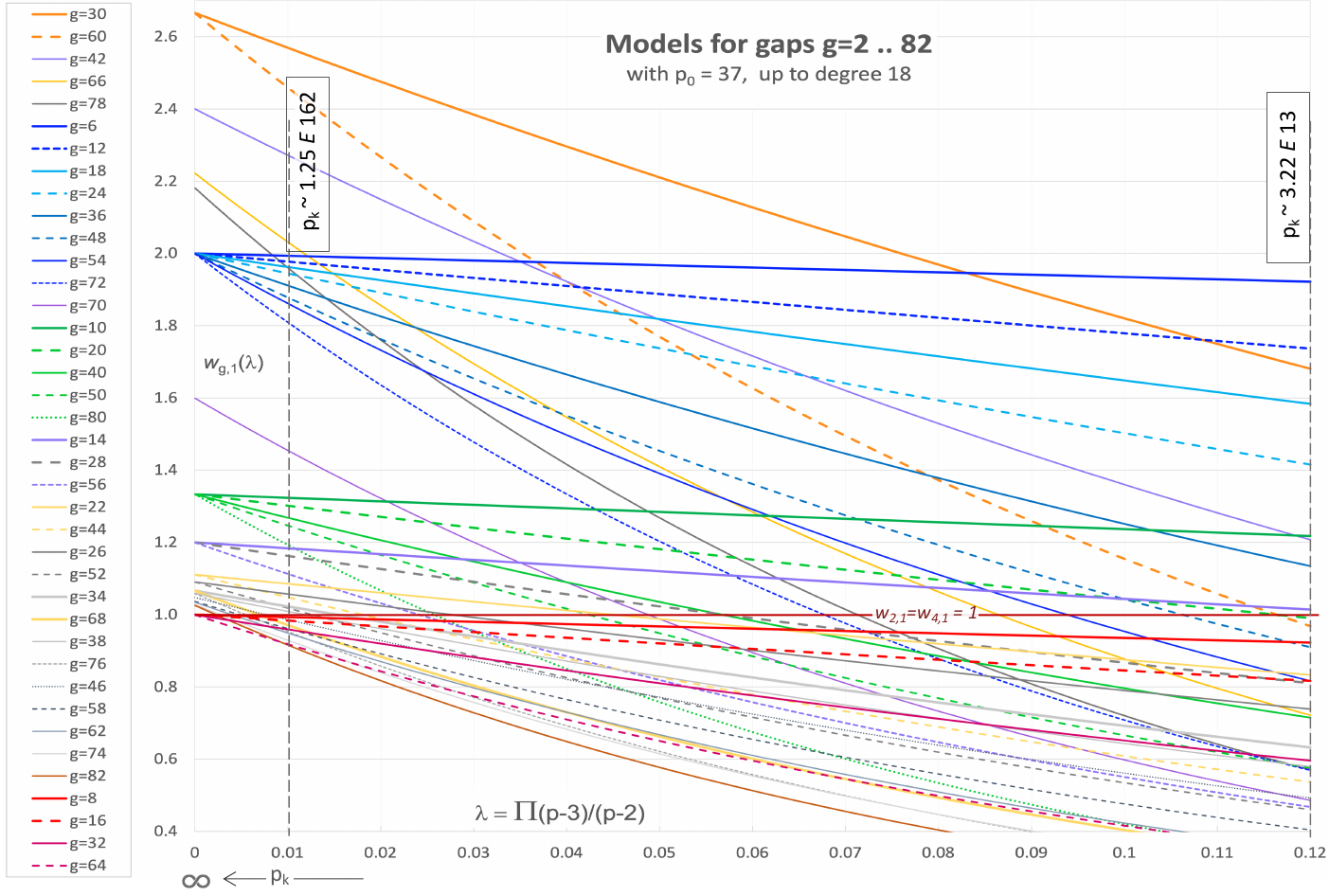}
\caption{\label{AllGapsZoomFig} Shown here is a closeup of the relative population models $w_{g,1}(\pml{p_k})$ for $p_0=37$ 
and all gaps $g \le 82$, as $\lambda$ gets small and $p_k$ gets large.  By $\lambda=0.01$ the populations of this set of gaps have almost
finished sorting into their asymptotic order.}
\end{figure}

From our observations above about the asymptotic ratios for small gaps, we believe that Lemke Oliver and
Soundararajan are observing transient phenomena.  Their sampled enumerations for $x_0=10^{12}$ are consistent
with $\lambda=0.27$ in Figure~\ref{WsFig}, and we see the ratios closing, especially for $W_0$.

These transient
biases will persist well beyond any computationally tractable primes.  Figure~\ref{AllGapsZoomFig} shows a closeup
of the relative population models $w_{g,1}(\lambda)$ for small $\lambda$ and large $p_k$, well beyond the computational range.

\section{Distributions in other bases}\label{SecBases}
The work above has addressed the residue classes of primes in base $10$.  For base $10$ 
the population models for gaps across stages of Eratosthenes sieve suggest
that the biases calculated by Lemke Oliver and 
Soundararajan are transient phenomena.  These biases will gradually fade for very large primes.  
How is this analysis affected by the choice of base?

Our work in base $10$ consisted of three components:  identifying the ordered pairs of last digits $(a,b)$ and gaps $g$
that correspond to each residue class; for a parameter $g_{\max}$,
comparing the asymptotic ratios $w_{g,1}(\infty)$ for the gaps $g \le g_{\max}$ within each 
residue class; and looking at the initial conditions and rates of convergence for the gaps within each residue class.
As we consider other bases, we look at the effect that a new base has on each of these components.

By setting a base, we set the assignment of gaps, especially the small gaps, to the respective residue classes.
These residue classes inherit the initial biases and rates of convergence associated with the assigned gaps.

We again emphasize that the initial populations are dominated by small gaps, especially the gaps $g \le 30$.  
Figure~\ref{AllGapsFig} illustrates the components of the resulting bias.  
These biases in the initial populations for small gaps will be inherited by the residue classes to which the small gaps belong.  

\vskip 0.125in

We illustrate this assignment of the initial bias by considering the small gaps $2 \le g \le 82$ 
under the bases $3$, $8$, and $30$.  

Since the gaps are all even, for any odd base $N$ there is a one-to-one map
from its residue classes into the base $2N$, that preserves the assignments of gaps to their 
residue classes.  For example, the results for bases $3$ and $6$ are equivalent, and the results for bases
$5$ and $10$ are equivalent.

The base $3$ acts much like the base $10$, in that
a prominent family of gaps is assigned to a single residue class.  For base $10$ all the
gaps $g=5n$ are assigned to $0 \bmod 10$, and for base $3$ all of the gaps $g=3n$ are
assigned to $0 \bmod 3$.  For the base $8$, as a multiple of $2$, the gaps with odd prime factors
spread out amongst the residue classes.

The base $30$ provides a broader distribution of the small gaps, including the multiples
 of $6$ and $10$.  The initial biases will strongly favor a small set of residue classes, but
 the asymptotics will eventually restore the balance.

\subsection{Distributions in base $3$.}  Lemke Oliver and Soundararajan calculated the distributions of the pairs of
last digits $(a,b)$ of primes modulo $3$ up through $10^{12}$, and they compare these favorably to a
conjectured model derived
from the Hardy and Littlewood's work on the $k$-tuple conjecture \cite{HL}.

In our approach through the cycles $\pgap(\pml{p})$, the ordered pairs $(a,b)$ of last digits correspond
to residue class ${b-a = r \bmod 3}$, and each gap $g$ is assigned to
its residue class $r \bmod 3$.

\begin{center}
\begin{tabular}{clclc}
$r \bmod 3$ &  \multicolumn{1}{c}{small $g$'s} & \multicolumn{1}{c}{$(a,b)$'s base $3$} & 
 $\Leftrightarrow $ &  \multicolumn{1}{c}{expected $W_r$} \\ \hline
$2$ & $2, 8, 14, 20, 26, \ldots $ & $(2,1)$ & & $1$ \\
$1$ &  $4, 10, 16, 22, 28,  \ldots $ & $(1,2)$ & & $1$ \\
$0$ &  $6, 12, 18, 24, 30, \ldots $ & $(1,1), \; (2,2)$ & & $2$ \\
\end{tabular}
\end{center}

Assuming that all of the ordered pairs eventually occur equally often, we set simple expected 
values for the $W_r$, setting $W_2=1$.  We list Lemke Oliver and Soundararajan's partial data for
$x_0=10^9$ and $x_0=10^{12}$,
and we compare this to $W_r(g_{\max},\pml{p})$ at $\lambda=0.360$ and $\lambda=0.270$ respectively. 
For $g_{\max}$
we have models up to $g=82$, but we use $g_{\max}=60$ to balance the contributions of the
gaps that are multiples of $3$ or $5$ across the residue classes.  For the limits $W_r(g_{\max},\infty)$
we don't require the models for the gaps, just their prime factors and associated weight
$w_{g,1}(\infty)$.  We tabulate $W_r(60,\infty)$ and $W_r(30030,\infty)$.

\begin{center}
\begin{tabular}{c|cc|cc|cc}
  & \multicolumn{6}{c}{$W_r(g_{\max},\pml{p})$ in base $3$} \\
$r \bmod 3$ & OS $10^9$ & $W_r (60,  )$ & OS $10^{12}$ & $W_r (60,  )$ & $W_r(60,\infty)$ &
 $W_r({\lil 30030},\infty)$  \\ 
  & & $\lil \lambda = 0.360$ & & $\lil \lambda = 0.270$ & & \\ \hline
$W_2$ & & $1.0000$ & & $1.0000$ & $1.00000$ & $1.00000$ \\
$W_1$ & $1.0000$ & $0.9991$ & $1.0000$ & $0.9981$ & $0.99051 $ & $1.00001$ \\
$W_0$ & $1.6045$ & $1.5864$ & $1.6863$ & $1.6641$ & $1.91863 $ & $1.99958$ \\
\end{tabular}
\end{center}

In base $3$ all multiples of $6$ will fall in the class $h = 0 \bmod 6$,
and thus all of the primorials $g=\pml{p}$ will fall within this class. 
In $\pgap(\pml{11})$ the gap $g=6$ is the most frequent gap, and it grows more quickly than other gaps for many more
stages of the sieve. 

For base $3$ the biases are somewhat muted through the first $100$ million primes, primarily affecting $W_0$.  As in base $10$
the biases have disappeared among the asymptotic values for $g_{\max}=30030$.

\subsection{Distributions in base $8$.} 
In base $8$ the small gaps are distributed more evenly across the residue classes.  We observe that 
the residue class $h = 0 \bmod 8$ starts slowly, and it lags the other classes for a long time.
Since the base is a power of $2$, every odd prime circulates through the residue classes.

\begin{center}
\begin{tabular}{clclc}
$r \bmod 8$ &  \multicolumn{1}{c}{small $g$'s} & \multicolumn{1}{c}{$(a,b)$'s base $8$} & 
 $\Leftrightarrow $ &  \multicolumn{1}{c}{expected $W_r$} \\ \hline
$2$ & $2, 10, 18, 26, \ldots $ & $\lil (1,3), \; (3,5), \; (5,7), \; (7,1)$ & & $1$ \\
$4$ & $4, 12, 20, 28,  \ldots $ & $\lil (1,5), \; (3,7), \; (5,1), \; (7,3)$ & &  $1$ \\
$6$ & $6, 14, 22, 30, \ldots $ & $\lil (1,7), \; (3,1), \; (5,3), \; (7,5)$ & & $1$ \\
$0$ & $8, 16, 24, 32, \ldots $ & $\lil (1,1), \; (3,3), \; (5,5), \; (7,7)$ & & $1$ \\
\end{tabular}
\end{center}

With base $8$, we use the models up to $g_{\max}=80$.  This distributes the multiples of $5$
evenly across the residue classes, although the class $W_6$ has one additional multiple of $3$, 
the gap $g=78$.

We compare Lemke Oliver and Sandararajan's data to $W_r(80, )$ for $\lambda=0.360$ and $0.270$, and
we list the limits $W_r(g_{\max},\infty)$ for $g_{\max} = 80$ and $30030$.

\begin{center}
\begin{tabular}{c|cc|cc|cc}
  & \multicolumn{6}{c}{$W_r(g_{\max},\pml{p})$ in base $8$} \\
$r \bmod 8$ & OS $10^9$ & $W_r (80,  )$ & OS $10^{12}$ & $W_r (80,  )$ & $W_r(80,\infty)$ &
 $W_r({\lil 30030},\infty)$  \\ 
  & & $\lil \lambda = 0.360$ & & $\lil \lambda = 0.270$ & & \\ \hline
$W_2$ & $1.0000$ & $1.0000$ & $1.0000$ & $1.0000$ & $1.00000$ & $1.00000$ \\
$W_4$ & $0.9585$ & $0.9552$ & $0.9708$ & $0.9671$ & $1.00106 $ & $0.99999$ \\
$W_6$ & $1.0034$ & $1.0056$ & $1.0059$ & $1.0147$ & $1.09568 $ & $1.00035$ \\
$W_0$ & $0.6739$ & $0.6599$ & $0.7395$ & $0.7258$ & $0.95553 $ & $0.99928$ \\
\end{tabular}
\end{center}

\subsection{Distributions in base $30$.}
The next primorial base is $30 = \pml{5}$.  This base is big enough that the small gaps are well separated, and the multiples
of $g=6$ and $g=10$ fall into a few distinct classes.  The early bias toward small gaps $g=2,4,6,10,12$ and 
even $g=14,18$ fall into separate residue classes.  Table~\ref{B30WTable} shows the early biases in $W_r(60,\pml{p})$ for base $30$, 
alongside the asymptotic values for $W_r(60,\infty)$ and $W_r(30030,\infty)$.

\begin{table}
\begin{center}
\begin{tabular}{clcc} 
$r \bmod {\bf 30}$ & $g$'s & $(a,b)$ & expected $W_r$ \\ \hline
$\lil W_2$ & $\lil 2, 32, \ldots$ & $\llil (29,1), (11,13), (17,19)$ &  $\lil 1$  \\
$\lil W_4$ & $\lil 4, 34, \ldots$ & $\llil (7,11), (13,17), (19,23)$ &  $\lil 1$ \\
$\lil W_6$ & $\lil 6, 36, \ldots$ & $\llil (1,7), (7,13), (13,19), (11,17), (17,23), (23,29)$ &  $\lil 2$ \\
$\lil W_8$ & $\lil 8, 38, \ldots$ & $\llil (11,19), (23,1), (29,7)$ &  $\lil 1$ \\
$\lil W_{10}$ & $\lil 10, 40, \ldots$ & $\llil (1,11), (7,17), (13,23), (19,29)$ &  $\lil 4/3$ \\
$\lil W_{12}$ & $\lil 12, 42, \ldots$ & $\llil (1,13), (7,19), (11,23), (17,29), (19,1), (29,11)$ &  $\lil 2$ \\
$\lil W_{14}$ & $\lil 14, 44, \ldots$ & $\llil (17,1), (23,7), (29,13)$ &  $\lil 1$ \\
$\lil W_{16}$ & $\lil 16, 46, \ldots$ & $\llil (1,17), (7,23), (13,29)$ &  $\lil 1$ \\
$\lil W_{18}$ & $\lil 18, 48, \ldots$ & $\llil (1,19), (11,29), (13,1), (19,7), (23,11), (29,17)$ &  $\lil 2$ \\
$\lil W_{20}$ & $\lil 20, 50, \ldots$ & $\llil (11,1), (17,7), (23,13), (29,19)$ &  $\lil 4/3$ \\
$\lil W_{22}$ & $\lil 22, 52, \ldots$ & $\llil (1,23), (7,29), (19,11)$ &  $\lil 1$ \\
$\lil W_{24}$ & $\lil 24, 54, \ldots$ & $\llil (7,1), (13,7), (19,13), (17,11), (23,17), (29,23)$ &  $\lil 2$ \\
$\lil W_{26}$ & $\lil 26, 56, \ldots$ & $\llil (11,7), (17,13), (23,19)$ &  $\lil 1$ \\
$\lil W_{28}$ & $\lil 28, 58, \ldots$ & $\llil (1,29), (13,11), (19,17)$ &  $\lil 1$ \\
$\lil W_0$ & $\lil 30, 60, \ldots$ & $\llil (1,1), (7,7), (11,11), (13,13),$ & $\lil 8/3$ \\ 
 & & $\llil (17,17), (19,19), (23,23), (29,29)$ & \\ 
\end{tabular}
\caption{\label{B30Table} For base $30$ 
we look at the last digits $(a,b)$ of 
consecutive primes, sort these into residue classes $r \bmod 30$, and 
identify the associated gaps for each class.  Assuming that each ordered pair
$(a,b)$ is equally likely to occur and setting $W_2=1$, we list the expected value
of $W_r$.}
\end{center}
\end{table}

\begin{table}
\begin{center}
\begin{tabular}{|c|cccc|} \hline
$\lil r \bmod {\bf 30}$ & $\lil W_r(60, {\lil \lambda=0.348})$ & $\lil W_r(60,\infty)$ & $\lil W_r(30030, \infty)$ & 
   $\lil {\rm expected} \; W_r$ \\ \hline
$\lil W_2$ & $\lil 1$ & $\lil 1$ &  $\lil 1$ &  $\lil 1$ \\
$\lil W_4$ & $\lil 1.00942$ & $\lil 1.03333$ & $\lil 1.00003$ & $\lil 1$ \\
$\lil W_6$ & $\lil 1.77813$ & $\lil 2$ & $\lil 1.99988$ & $\lil 2$ \\
$\lil W_8$ & $\lil 0.79999$ & $\lil 1.02941$ & $\lil 0.99998$ & $\lil 1$ \\
$\lil W_{10}$ & $\lil 1.02067$ & $\lil 1.33333$ & $\lil 1.33317$ & $\lil 4/3$ \\
$\lil W_{12}$ & $\lil 1.35454$ & $\lil 2.2$ & $\lil 1.99989$ & $\lil 2$ \\
$\lil W_{14}$ & $\lil 0.71359$ & $\lil 1.15556$ & $\lil 0.99991$ & $\lil 1$ \\
$\lil W_{16}$ & $\lil 0.54959$ & $\lil 1.02381$ & $\lil 1.00002$ & $\lil 1$ \\
$\lil W_{18}$ & $\lil 0.99517$ & $\lil 2$ & $\lil 1.99996$ & $\lil 2$ \\
$\lil W_{20}$ & $\lil 0.55173$ & $\lil 1.33333$ & $\lil 1.33329$ & $\lil 4/3$ \\
$\lil W_{22}$ & $\lil 0.46397$ & $\lil 1.10101$ & $\lil 1.00004$ & $\lil 1$ \\
$\lil W_{24}$ & $\lil 0.71723$ & $\lil 2$ & $\lil 1.99983$ & $\lil 2$ \\
$\lil W_{26}$ & $\lil 0.34735$ & $\lil 1.14545$ & $\lil 0.99997$ & $\lil 1$ \\
$\lil W_{28}$ & $\lil 0.36493$ & $\lil 1.11852$ & $\lil 0.99994$ & $\lil 1$ \\
$\lil W_0$ & $\lil 0.67385$ & $\lil 2.66667$ & $\lil 2.66451$ & $\lil 8/3$ \\ \hline
\end{tabular}
\caption{\label{B30WTable} The table for 
the distribution of $W_r(g_{\max},\pml{p})$ in base $30$.  In a base this large, biases across the
residue classes will persist until we give large gaps ample time to show up in the cycles $\pgap(\pml{p})$.}
\end{center}
\end{table}

\section{Conclusion}
For the first $10^8$ primes, Lemke Oliver and Soundararajan \cite{OS,OSQ} calculated how often the possible pairs
$(a,b)$ of last digits of consecutive primes occurred, and they observed biases compared to simple expected values.  
Regarding their calculations
they raised two questions: Does the observed bias persist?  Is the observed bias dependent upon the base? 
We have addressed  both of these questions by using the dynamic system that exactly models the populations
of gaps across stages of Eratosthenes sieve. 

The study of last digits of consecutive primes provides an excellent example with which to
contrast the primes we can observe versus the primes in the large.
The initial biases and more
rapid convergence favor the small gaps over any computationally tractable range,
and these give a substantial head start to the residue classes to which these small
gaps are assigned.

The observed biases are transient phenomena, but they persist well past the range of 
computationally tractable 
primes.  
The observed biases are due to the quick appearance of small gaps and the slow evolution
of the dynamic system. The asymptotics of the dynamic system play out beyond any conceivable
computational horizon. 

To put this in perspective, 
the cycle $\pgap(\pml{199})$ has more gaps than there are
atoms in the known universe; yet in $\pgap(\pml{p})$ for a $16$-digit prime $p$, 
small gaps like $g=30$ will still be appearing
in frequencies well below their ultimate ratios.  
Running Eratosthenes sieve through $16$-digit primes, $\lambda \approx 0.096$, the three-digit gaps will just be emerging 
compared to the prevailing populations of smaller gaps.

The population models for gaps address the inter-class bias, the distributions across the
residue classes.
We have not addressed any intra-class bias, that is, any uneven distribution
across the ordered pairs $(a,b)$ within a given residue class.  
Once we understand the model for gaps, then any choice of base reassigns the gaps across the residue classes for this base.  The number of ordered
pairs corresponding to a residue class $r$ provides an expected value $W_r$,
and we have seen above that the asymptotic values $W_r(g_{\max},\infty)$ approach these
expected values.

Initial observations along these lines were made in 2016 \cite{FBHLastD}.  Here we advance the 2016 approach with more recent insights
into the evolution of gaps across stages of Eratosthenes sieve.


\bibliographystyle{amsplain}


\begin{thebibliography}{10}

\bibitem{BrentSmall}
R.P. Brent, \emph{The distribution of small gaps between successive primes}, Math. of Computation, 28(125), Jan 1974.

\bibitem{Brown}
S. Brown, \emph{Distance between consecutive elements of the multiplicative group of integers modulo $n$}, Notes on Number Theory
and Discrete Mathematics, 30(1), Feb 2024.

\bibitem{FBHSFU}
F.B. Holt \& H. Rudd, \emph{Combinatorics of the gaps between primes}, Connections in Discrete Mathematics, Simon Fraser U.,
arXiv 1510.00743, June 2015.

\bibitem{FBHPatterns}
F.B. Holt, \emph{Patterns among the Primes}, KDP, June 2022.

\bibitem{FBHLastD}
F.B. Holt, \emph{On the last digits of consecutive primes}, arXiv:1604.02443, July 2016.

\bibitem{FBH2p}
F.B. Holt, \emph{Addendum: models for gaps $g=2p_1$}, arXiv:2309.16833v1, Sept 2023.

\bibitem{HL}
G.H. Hardy and J.E. Littlewood, \emph{Some problems in 'partitio numerorum'
  iii: On the expression of a number as a sum of primes}, G.H. Hardy Collected
  Papers, vol.~1, Clarendon Press, 1966, pp.~561--630.

\bibitem{OSQ}
E. Klarreich, \emph{Mathematicians discover prime conspiracy}, Quanta Magazine, March 2016, 

\bibitem{OS}
R. Lemke Oliver \& K. Soundararajan, \emph {Unexpected biases in the distribution of consecutive primes}, 
Proc Natl Acad Sci USA, 113(31);E4446-54, Aug 2016.

\end{thebibliography}
\providecommand{\bysame}{\leavevmode\hbox to3em{\hrulefill}\thinspace}
\providecommand{\MR}{\relax\ifhmode\unskip\space\fi MR }
\providecommand{\MRhref}[2]{%
  \href{http://www.ams.org/mathscinet-getitem?mr=#1}{#2}
}
\providecommand{\href}[2]{#2}

\end{document}